\documentclass[11pt,reqno,section]{amsart}
\usepackage[english]{babel}
\usepackage[utf8]{inputenc}
\usepackage{amsmath}
\usepackage{amsthm}
\usepackage{verbatim}
\usepackage{enumerate}
\usepackage{amssymb}
\usepackage{amsfonts}
\usepackage{amscd,bezier}
\usepackage{anysize}
\usepackage{indentfirst}
\usepackage{upref}
\usepackage{color}
\usepackage[pagewise]{lineno}
\usepackage{tikz}
\usepackage{theoremref}
\usetikzlibrary{matrix}
\usepackage{pstricks}
\usepackage{subfig}
\usepackage{graphicx}
\usepackage{pstricks}
\usepackage{amscd}
\usepackage{relsize}
\usepackage[toc,page]{appendix}
\usepackage{commath}
\usepackage{mathrsfs}
\usepackage[hidelinks]{hyperref}

\allowdisplaybreaks

\numberwithin{equation}{section}

\newtheorem{theorem}{Theorem}[section]
\newtheorem{lemma}[theorem]{Lemma}

\newtheorem{definition}[theorem]{Definition}
\newtheorem{proposition}[theorem]{Proposition}
\newtheorem{corollary}[theorem]{Corollary}

\newtheorem{example}[theorem]{Example}

\newtheorem{remark}[theorem]{Remark}

\marginsize{3cm}{3cm}{3cm}{3cm}

\newcommand{\NmuD}{{\mathrm{N}\mu\mathrm{D}}}

\newcommand{\walpha}{{\widetilde{\alpha}}}
\newcommand{\wbeta}{{\widetilde{\beta}}}

\newcommand{\muD}{{\mu\mathrm{D}}}

\newcommand{\Nmu}{{\mathrm{N}\mu}}

\renewcommand{\S}{\mathscr{S}}
\newcommand{\U}{\mathscr{U}}
\newcommand{\B}{\mathscr{B}}
\newcommand{\V}{\mathscr{V}}
\newcommand{\vPhi}{\widehat{\Phi}}
\newcommand{\F}{(2.2)_{\tau\in \mathbb{R}}}
\newcommand{\FT}{(2.9)_{\tau\in \mathbb{R}}}

\newcommand{\W}{\mathscr{W}}

\newcommand{\Si}{\Sigma}

\newcommand{\Id}{\mathrm{Id}}

\newcommand{\rang}{{\mathrm{range}\,}}

\renewcommand{\P}{\mathrm{P}}

\newcommand{\sgn}{\mathrm{sgn}}

\newcommand{\qeg}{\mathrm{q}}
\newcommand{\ceg}{\mathrm{c}}
\newcommand{\seg}{\mathrm{s}}
\newcommand{\ueg}{\mathrm{u}}

\newcommand{\A}{\mathbb{A}}

\renewcommand{\H}{\mathbb{H}}
\renewcommand{\O}{\mathbb{O}}

\numberwithin{equation}{section}

\allowdisplaybreaks
\usepackage{color}

\title[Nonexponential and nonuniform behavior on skew-products]{Dichotomies Faster or Slower than exponential are Irrelevant for Skew-Product Flows}

\author[N.~Jara]{N\'estor Jara}
\author[I.P.~Longo]{Iacopo P.~Longo}
\author[M.~Rasmussen]{Martin Rasmussen}
\address[N.~Jara]{Departamento de Matemática, Universidad Técnica Federico Santa María, Casilla 110-V, Valparaíso, Chile.}
\address[I.P.~Longo]{University of Exeter, Department of Mathematics and Statistics, 925 Laver Building, 23 N Park Rd, Exeter EX4 4QD, United Kingdom}

\address[M.~Rasmussen]{Imperial College London,
635 Huxley Building, 180 Queen’s Gate, London SW7 2AZ, UK}
\email[N.~Jara]{ nestor.jara@usm.cl} 

\email[I.P.~Longo]{i.longo@exeter.ac.uk} 

\email[M.~Rasmussen]{m.rasmussen@imperial.ac.uk} 

\thanks{
Iacopo P.~Longo acknowledges partial support from UKRI under the grant agreement
EP/X027651/1, from  MICIIN/FEDER project PID2021-125446NB-I00 and from the University of
Valladolid under project PIP-TCESC-2020. }

\date{}

\begin{document}
	
	\begin{abstract}
We prove that dichotomies given by growth rates that are either faster or slower than exponential either do not occur or are inconsequential in the setting of skew-products with compact base. A similar conclusion is obtained for the nonuniform exponential behavior. To achieve this, we study families of translated linear nonautonomous differential equations for which we prove propagation results. We also study translations of growth rates under a comparison criteria. 



	\end{abstract}
	

	

	\maketitle

\section{Introduction}

The study of linear nonautonomous differential equations has long been guided by characterizing exponential dichotomies and bounded growth \cite{Perron,Siegmund2002,Siegmund2}. These concepts are particularly effective in classical contexts like periodic and almost periodic systems, where tools such as Floquet theory \cite{Floquet,Russel} provide a deep understanding of long-term behavior. However, outside of these well-structured regimes, the dynamics can be significantly more complex, since time dependent systems can exhibit nonuniform or nonexponential growth and decay.

\smallskip

A central framework for studying such systems is the theory of skew-product flows, which allow nonautonomous dynamics to be embedded in an autonomous flow on a product space of the form $\Omega\times\mathbb{R}^N$, where the evolution in the fiber depends continuously on the base point in $\Omega$, and the set $\Omega$ is typically a collection of uniformly continuous and bounded functions \cite{Sacker,Sacker2} or almost periodic \cite{Fink}, and has later been expanded to collections of locally integrable maps endowed with a certain topology \cite{Artstein1,Artstein2,Longo,Longo2,Longo5}.

\smallskip

In recent years, research has moved beyond the classical exponential dichotomy to explore more general growth regimes. A notable generalization is the concept of nonuniform exponential dichotomy, introduced by Barreira and Valls \cite{BV-CMP, BV-JDE}. This framework has played a central role in the study of nonuniform hyperbolicity and has found applications in the analysis of functional, differential, and difference equations \cite{DZZ, DZZ-PL,LOO}. 

\smallskip

In parallel, the study of dichotomies governed by nonexponential growth behaviors has emerged as another line of generalization. To the best of our knowledge, the earliest contributions in this direction are due to Pinto and Naulín \cite{PN1, PN2, PN3}, whose work initiated a line of research that has continued to develop and has been applied in various contexts \cite{Bento, ZFY}, including investigations into polynomial behavior \cite{Dragicevic5}. A related approach is found in the work of Pötzsche \cite{Potzsche}, where the author introduces the notion of generalized exponential functions.

\smallskip

More recently, Silva \cite{Silva} proposed a unifying framework via the concept of nonuniform $\mu$-dichotomy—which includes the nonuniform exponential dichotomy as a special case—where, $\mu$ is a so-called growth rate (Def.~\ref{620}). This broader perspective has informed subsequent developments in \cite{GJ2}. These approaches provide a more general setting for analyzing stability and instability phenomena and naturally lead to the definition of a dichotomy spectrum, which generalizes the classical concept introduced in \cite{Siegmund2002}.

\smallskip

While the classical notion of exponential dichotomy has been extensively studied in the skew-product setting \cite{Dueñas, Longo}, to the best of our knowledge, the generalized notions of nonautonomous hyperbolicity discussed above—though developed in both discrete and continuous time settings, and in both finite- and infinite-dimensional spaces—have not yet been applied to the study of skew-product dynamics. In particular, regarding the spectral perspective, these generalizations have extended the dichotomy spectrum for nonautonomous behavior as described in \cite{Siegmund2002}, but have not incorporated the classical Sacker–Sell spectrum \cite{Sacker} for skew-products within their frameworks.

\smallskip

The present work aims to clarify how generalized notions of hyperbolicity interact with the structure of skew-products. Our main conclusion is that a spectral characterization of a given equation can be extended to its entire hull \textbf{if and only if} the system admits a uniform exponential dichotomy. In contrast, for a rather large class of nonexponential or nonuniform behaviors (see Rem.~\ref{extend}), such an extension is only possible along a single orbit. In other words, our findings indicate that the results in \cite{Dueñas, Longo, Sacker, Longo5} concerning exponential dichotomies on invariant sets are optimal with respect to the growth rate and are uniquely characteristic of uniform dichotomies. This means that nonuniform or nonexponential dichotomies and growth rates cannot, in general, be extended to larger invariant sets. This conclusion follows from Corollaries \ref{605}, \ref{617}, and \ref{616}.

\subsection{Structure of the article}
In Section \ref{dynamic} we introduce the concepts of dichotomy (Def.~\ref{99}) and bounded growth (Def.~\ref{134}) for families of translated equations. These notions capture both nonexponential and nonuniform dynamics. We demonstrate that these notions support propagation results (Lem.~\ref{44} \& \ref{101}); that is, when a given equation exhibits dichotomy or bounded growth, then the entire family of its translations inherits this behavior. In Subsection \ref{dynaspectrum} we define the spectrum associated with this notion of dichotomy and conclude that the spectrum of the family coincides with the dichotomy spectrum of each individual equation (Thm.~\ref{126}).

\smallskip

Section \ref{comparison} describes the behavior of translated growth rates. We define the slow and fast growth rates (Def.~\ref{608}), and prove that the point-wise limit of their translations are no longer growth rates (Lem.~\ref{401} \& \ref{402}). This leads to the conclusion that, in general, dichotomies and bounded growth for families of equations cannot be described by a unified growth rate. When such a unification is possible, it corresponds to the classical exponential case (Cor.~\ref{403}).

\smallskip

In Section \ref{behavior} we present the main conclusions of this work. We begin by introducing the skew-product formalism in Subsection \ref{skew}. Finally, in Section \ref{limit}, we apply the developments of the previous sections to analyze the behavior of limit points in the skew-product setting. Our main findings prove that for slow growth rates, limit points can never exhibit dichotomy (Thm.~\ref{601}), while  for fast growth rates, we prove that limit points do not exist (Thm.~\ref{602}), and consequently, such rates cannot occur when the base space is compact (Cor.~\ref{611}). In subsection \ref{periodic} for both cases we demonstrate that nonexponential behavior is, in a sense, incompatible with periodic or almost periodic functions (Prop.~\ref{128} \& \ref{613}). Similarly, for nonuniform exponential behavior we prove limiting equations do not exists (Cor.~\ref{616}).

\section{Propagation results for families of translated systems}\label{dynamic}

Consider a locally integrable map $A:\mathbb{R}\to \mathbb{R}^{N\times N}$ and the equation
\begin{equation}\label{01}
    \dot{x}=A(t)x(t),\qquad\forall\,t\in \mathbb{R},
\end{equation}
with evolution operator $\Phi:\mathbb{R}^2\to \mathbb{R}^{N\times N}$. For $\tau\in\mathbb{R}$, we define the \textbf{$\tau$-translation} of $A$, given by $A_\tau(t)=A(t+\tau)$, with which we define the translated equation from \eqref{01} as
\begin{equation}\label{00}
    \dot{x}=A_\tau(t)x(t),\qquad\forall\,t\in \mathbb{R}.
\end{equation}

Note that for a fixed $\tau\in \mathbb{R}$, the evolution operator for \eqref{00} is given by $\Phi_\tau(t,s)=\Phi(t+\tau,s+\tau)$ for every $t,s\in \mathbb{R}$. In particular, $\Phi_0=\Phi$.

\smallskip

Through this section we will consider all three: the original equation \eqref{01}, the translation of it by a fixed $\tau$ given by \eqref{00}, and also the family of all these translations, denoted by $\F$. The goal of this section is to prove that dynamical behavior (dichotomy, bounded growth, spectrum) for equation \eqref{01} is equivalent to the same behavior for the complete family of equations $\F$.

\smallskip

In order to study the different behaviors that the solutions to these equations may present, consider the following definition, adapted from \cite[p.~621]{Silva}.

\begin{definition}\label{620}
   We say a function $\mu:\mathbb{R}\to \mathbb{R}^+$ is a \textbf{growth rate} if it is strictly increasing, differentiable and verifies $\mu(0)=1$, $\lim_{t\to +\infty}\mu(t)=+\infty$ and $\lim_{t\to-\infty}\mu(t)=0$.
\end{definition}

To explore some examples, denote the sign of $t\in \mathbb{R}$ by $\sgn(t)$. The map $t\mapsto \exp(t)=e^t$ defines the \textbf{exponential} growth rate, while $p:\mathbb{R}\to\mathbb{R}^+$, $p(t)=(|t|+1)^{\sgn(t)}$ is the \textbf{polynomial} rate. 

\smallskip

We also define the family of \textbf{superexponential} growth rates, for $r\in (1,+\infty)$, as the growth rate given by $t\mapsto \seg_r(t)=e^{\sgn(t)\cdot |t|^r}$. In this family we have, for instance, the map $\qeg=\seg_2:\mathbb{R}\to \mathbb{R}^+$, $\qeg(t)=e^{\sgn(t)t^2}$, which we call the \textbf{quadratic} growth rate, and the map $\ceg=\seg_3:\mathbb{R}\to \mathbb{R}^+$, $\ceg(t)=e^{t^3}$, called the \textbf{cubic} rate.

\smallskip

On the other hand, we have the \textbf{subexponential} growth rates, for $r\in (0,1)$,  as the growth rate given by $t\mapsto \ueg_r(t)=e^{\sgn(t)\cdot \left((|t|+1)^r-1\right)}$. In this family we have, for instance, the map $t\mapsto e^{\sgn(t)(\sqrt{|t|+1}-1)}$, which we call the \textbf{square root} rate.

\smallskip

Given a growth rate $\mu$ and $\tau\in \mathbb{R}$, we can consider the map $\mu_\tau:\mathbb{R}\to \mathbb{R}^+$ given by $\mu_\tau(t)=\frac{\mu(\tau+t)}{\mu(\tau)}$. This is a new growth rate called the \textbf{$\tau$-translated} growth rate.  It is immediate that $(\mu_\tau)_{\hat{\tau}}=\mu_{\tau+\hat{\tau}}$ for every $\tau,\hat{\tau}\in \mathbb{R}$. Note that every translation from the exponential rate is still the same exponential rate, and moreover, the exponential rate, and its powers, are the only rates with this property.

\smallskip

The following concept \cite[p.~621]{Silva}, illustrates how the notion of growth rates enables us to generalize previously studied notions of hyperbolicity within the nonautonomous framework, {\it i.e.}, exponential dichotomy \cite{Perron,Siegmund2002}, and to encompass other nonuniform \cite{BV} or nonexponential behaviors that may arise in this context, such as polynomial decay \cite{Dragicevic5}, among others. 

\begin{definition}\label{dichotomynormal}
   Let $\mu:\mathbb{R}\to \mathbb{R}$ be a growth rate. We say equation \eqref{01} has \textbf{nonuniform $\mu$-dichotomy} $(\NmuD)$, if there exist an invariant projector $s\mapsto \P(s)\in \mathbb{R}^{N\times N}$, that is $$\P(t)\Phi(t,s)=\Phi(t,s)\P(s),\qquad\forall\,t,s,\in \mathbb{R},$$
   and constants $K\geq 1$, $\alpha<0$, $\beta>0$ and $\theta,\nu\geq 0$ such that $\alpha+\theta<0$, $\beta-\nu>0$ and 
       \[\left\{ \begin{array}{lc}
            \norm{\Phi(t,s)\P(s)}&\leq K\left(\frac{\mu(t)}{\mu(s)}\right)^\alpha\mu(s)^{\sgn(s)\theta}\,\,\,\quad\text{ for }t\geq s,\nonumber\\
            \norm{\Phi(t,s)[\Id-\P(s)]}&\leq K\left(\frac{\mu(t)}{\mu(s)}\right)^\beta\mu(s)^{\sgn(s)\nu}\qquad\text{ for }t\leq s.
            \end{array}
           \right.\]
           If $\theta = \nu = 0$, then we say \eqref{01} admits uniform \textbf{$\mu$-dichotomy} $(\muD)$.
\end{definition}

In the case of $\NmuD$, the above estimates are referred to as a $\Nmu$-dichotomy with \textbf{parameters} $(\P; K, \alpha, \beta, \theta, \nu)$, whereas for $\muD$, they are referred to as a $\mu$-dichotomy with parameters $(\P; K, \alpha, \beta)$. Note that if $\P \equiv \Id$, then the parameters $\beta$ and $\nu$ are omitted; similarly, if $\P \equiv 0$, the parameters $\alpha$ and $\theta$ are not included. In both cases, we retain the notation and indicate the missing parameter(s) with an asterisk ($*$).

\smallskip

An immediate but important observation is that if $s\mapsto \P(s)$ is an invariant projector for \eqref{01}, then $s\mapsto \P_\tau(s):=\P(s+\tau)$ is an invariant projector for the $\tau$-translated equation \eqref{00}. The following result shows how a dichotomy for equation \eqref{01} can be propagated to the translated equation \eqref{00}.

\begin{lemma}\label{44}
If \eqref{01} has $\Nmu$-dichotomy with parameters $(\P; K,\alpha, \beta,\theta,\nu)$, then, for every $\tau\in \mathbb{R}$, the system  \eqref{00} has $\Nmu_\tau$-dichotomy with parameters $(\P_\tau;K_\tau, \alpha, \beta,\theta,\nu)$, where $\mu_\tau$ is the $\tau$-translated growth rate, the invariant projector is $s\mapsto \P_\tau(s):=\P(s+\tau)$ and
\begin{equation}\label{60}
    K_\tau:=K\cdot \mu(\tau)^{3\sgn(\tau)\cdot\max\{\theta,\nu\}}.
\end{equation}
\end{lemma}

\begin{proof}
We have
\begin{itemize}
    \item $\Phi_\tau(t,s)\P_\tau(s)=\Phi_\tau(t+\tau,s+\tau)\P(s+\tau)$,
    \item $\Phi_\tau(t,s)[\Id-\P_\tau(s)]=\Phi(t+\tau,s+ \tau)[\Id-\P(s+\tau)]$,
\end{itemize}
thus, as \eqref{01} has $\NmuD$ with parameters $(\P; K,\alpha, \beta,\theta,\nu)$, we obtain
    \begin{equation*}
        \left\{ \begin{array}{lc}
            \norm{\Phi_\tau(t,s)\P_\tau(s)}&\leq K\left(\frac{\mu(t+\tau)}{\mu(s+\tau)}\right)^\alpha\mu(s+\tau)^{\sgn(s+\tau)\theta}\,\,\,\quad\text{ for } t\geq s,\nonumber\\
            \norm{\Phi_\tau(t,s)[\Id-\P_\tau(s)]}&\leq K\left(\frac{\mu(t+\tau)}{\mu(s+\tau)}\right)^\beta\mu(s+\tau)^{\sgn(s+\tau)\nu}\qquad\text{ for }t\leq s.
            \end{array}
           \right.
    \end{equation*}

Observe that for any $a,b\in \mathbb{R}$ we have
\begin{align*}
    \left(\frac{\mu(t+\tau)}{\mu(s+\tau)}\right)^a\mu(s+\tau)^{\sgn(s+\tau)b}&=\left(\frac{\mu_\tau(t)}{\mu_\tau(s)}\right)^a\mu_\tau(s)^{\sgn(s)b}\left(\frac{\mu(s+\tau)^{\sgn(s+\tau)}}{\mu_\tau(s)^{\sgn(s)}}\right)^b.
\end{align*}

On the other hand, note that
\begin{align*}
    \frac{\mu(s+\tau)^{\sgn(s+\tau)}}{\mu_\tau(s)^{\sgn(s)}}    =\frac{\mu(s+\tau)^{\sgn(s+\tau)}\mu(\tau)^{\sgn(s)}}{\mu(s+\tau)^{\sgn(s)}}&\leq\mu(\tau)^{\sgn(\tau)}\cdot\mu(s+\tau)^{\sgn(s+\tau)-\sgn(s)}\\
    &\leq \mu(\tau)^{3\sgn(\tau)}.
\end{align*}

Therefore, defining $K_\tau$ as in \eqref{60} we have
    \begin{equation*}
        \left\{ \begin{array}{lc}
            \norm{\Phi_\tau(t,s)\P_\tau(s)}&\leq K_\tau\left(\frac{\mu_\tau(t)}{\mu_\tau(s)}\right)^\alpha\mu_\tau(s)^{\sgn(s)\theta}\,\,\,\quad\text{ for }t\geq s,\nonumber\\
            \norm{\Phi_\tau(t,s)[\Id-\P_\tau(s)]}&\leq K_\tau\left(\frac{\mu_\tau(t)}{\mu_\tau(s)}\right)^\beta\mu_\tau(s)^{\sgn(s)\nu}\qquad\text{ for }t\leq s,
            \end{array}
           \right.
    \end{equation*}
{\it i.e.}, the equation \eqref{00} has $\Nmu_\tau$-dichotomy with parameters $(\P_\tau;K_\tau,\alpha,\beta,\theta,\nu)$.
\end{proof}

\begin{remark}\label{49}
{\rm
In the previous result, $K_\tau$ varies continuously with respect to $\tau \in \mathbb{R}$. Moreover, if the dichotomy for the system \eqref{01} is uniform—that is, if $\theta = \nu = 0$—then the dichotomy for \eqref{00} is also uniform, and in this case $K_\tau= K$ for all $\tau \in \mathbb{R}$.
    }
\end{remark}

The following definition aims to emulate how the use exponential dichotomy has been used to describe the behavior of families of translated systems, but instead employs the concept of growth rates and takes into account the conclusion of Lem.~\ref{44}. 

\begin{definition}\label{99}
We say that the \textbf{family of equations $\F$ admits nonuniform $\mu$-dichotomy} $(\NmuD)$, if there is an invariant projector for \eqref{01} given by $s\mapsto \P(s)$, a continuous map $K:\mathbb{R}\to [1,+\infty), \tau\mapsto K_\tau$, and constants $\alpha<0$, $\beta>0$ and $\theta,\nu\geq 0$ verifying $\alpha+\theta<0$ and $\beta-\nu>0$ such that for every $\tau\in \mathbb{R}$ we have
       \[\left\{ \begin{array}{lc}
            \norm{\Phi_\tau(t,s)\P_\tau(s)}&\leq K_\tau\left(\frac{\mu_\tau(t)}{\mu_\tau(s)}\right)^\alpha\mu_\tau(s)^{\sgn(s)\theta}\,\,\,\quad\text{ for }t\geq s,\nonumber\\
            \norm{\Phi_\tau(t,s)[\Id-\P_\tau(s)]}&\leq K_\tau\left(\frac{\mu_\tau(t)}{\mu_\tau(s)}\right)^\beta\mu_\tau(s)^{\sgn(s)\nu}\qquad\text{ for }t\leq s.
            \end{array}
           \right.\]
            If $\theta=\nu=0$ and $K$ is constant, then we say \textbf{the family of equations $\F$ admits uniform $\mu$-dichotomy} $(\muD)$.
\end{definition}

We maintain the notation of parameters for dichotomies of families of equations, except that in this context $K$ is a map rather than a constant. An immediate corollary follows from Lem.~\ref{44} and Rem.~\ref{49}.

\begin{corollary}\label{200}
The equation \eqref{01} has $\NmuD$ (resp. $\muD$) if and only if the family of equations $\F$ has $\NmuD$ (resp. $\muD$).
\end{corollary}

The notions of dichotomy can be interpreted as measures of how rapidly the evolution of a system contracts or expands. The following concept, adapted from \cite[p.~630]{Silva}, may be regarded as a complementary notion, capturing how slowly the system's evolution undergoes contraction or expansion

\begin{definition}
	Consider a growth rate $\mu$.	We say \eqref{01} has \textbf{nonuniform $\mu$-bounded growth}, or just $\Nmu$-growth, if there are constants $L\geq 1$, $a>0$ and $\epsilon\geq0$ such that
\begin{equation*}
    		\|\Phi(t,s)\|\leq L\left(\frac{\mu(t)}{\mu(s)}\right)^{\sgn(t-s)a}\mu(s)^{\sgn(s)\epsilon},\quad\forall\,t,s\in \mathbb{R}.
\end{equation*}
Moreover, if $\epsilon=0$, we say\eqref{01} has \textbf{uniform $\mu$-bounded growth}, or just $\mu$-growth.
	\end{definition}

    The next result gives an analogous conclusion to Lem.~\ref{44} for the concept of $\Nmu$-growth.

\begin{lemma}\label{101}
    Suppose \eqref{01} has $\Nmu$-growth with constants $L\geq 1$, $a>0$ and $\epsilon\geq 0$. Then, for every $\tau\in \mathbb{R}$, the system \eqref{00} has $\Nmu_\tau$-growth with constants $L_\tau:=L\cdot\mu(\tau)^{3\sgn(\tau)\epsilon}$, $a>0$ and $\epsilon\geq 0$, where $\mu_\tau$ is the $\tau$-translated growth rate.
\end{lemma}

\begin{proof}
    It is enough to note that 
    \begin{align*}
        \|\Phi_\tau(t,s)\|&\leq L\left(\frac{\mu(t+\tau)}{\mu(s+\tau)}\right)^{\sgn(t+\tau-(s+\tau))a}\mu(s+\tau)^{\sgn(s+\tau)\epsilon}\\
        &\leq L_\tau\left(\frac{\mu_\tau(t)}{\mu_\tau(s)}\right)^{\sgn(t-s)a}\mu_\tau(s)^{\sgn(s)\epsilon},\quad\forall\,t,s,\tau\in \mathbb{R},
    \end{align*}
    where the last inequality is deduced in the same fashion as in Lem.~\ref{44}.
\end{proof}

\begin{remark}\label{304}
    {\rm
    Once again, in the result above, $L_\tau$ varies continuously with respect to $\tau \in \mathbb{R}$. Moreover, if we fix $\epsilon = 0$, we recover the uniform growth case. 
    }
\end{remark}

The previous lemma motivates the following definition.

\begin{definition}\label{134}
Consider a growth rate $\mu$. We say that \textbf{the family of equations $\F$ has nonuniform $\mu$-bounded growth}, or just $\Nmu$-growth, if there exist a continuous function $L:\mathbb{R}\to [1,+\infty)$, $\tau\mapsto L_\tau$, and constants $a>0$, $\epsilon\geq 0$ such that 
\begin{equation*}
    		\|\Phi_\tau(t,s)\|\leq L_\tau\left(\frac{\mu_\tau(t)}{\mu_\tau(s)}\right)^{\sgn(t-s)a}\mu_\tau(s)^{\sgn(s)\epsilon},\quad\forall\,t,s,\tau\in \mathbb{R}.
\end{equation*}
Moreover, if $\epsilon=0$ and $L$ is constant, it is said that \textbf{the family of equations $\F$ has uniform $\mu$-bounded growth}, or just $\mu$-growth.
\end{definition}

Once again, an immediate corollary is deduced from Lem.~\ref{101} and Rem.~\ref{304}.

\begin{corollary}\label{201}
    The equation \eqref{01} has $\Nmu$-growth (resp. $\mu$-growth) if and only if the family of equations $\F$ has $\Nmu$-growth (resp. $\mu$-growth).
\end{corollary}

Let us illustrate these concepts with examples.

\begin{example}\label{550}
    {\rm Consider the family of equations
    \begin{equation}\label{55}
    \dot{x}=\frac{1}{1+|t+\tau|}x(t),\qquad\tau\in \mathbb{R}.
\end{equation}

For $\tau = 0$, the equation is $\dot{x} = \frac{1}{1 + |t|} x(t)$, whose evolution operator is $\Phi(t,s) = \frac{(1 + |t|)^{\sgn(t)}}{(1+|s|)^{\sgn(s)}}$. Therefore, it exhibits a uniform polynomial dichotomy with parameters $(0; 1, *, 1)$. Moreover, it also exhibits uniform polynomial growth with constants $L=1$ and $a=1$. Hence, by Cor.~\ref{200} and Cor.~\ref{201}, the family of equations \eqref{55} possesses both, polynomial dichotomy and growth.
    }
\end{example}

In the previous example, the notions of dichotomy and growth obtained for the family of systems require that the growth rate changes for each translated system. Later we discuss the possibility of writing this kind of behavior under a unified growth rate (see Cor.~\ref{403}). A similar conclusion is obtained in the following example.

\begin{example}\label{056}
    {\rm
   Consider the family of differential equations
\begin{equation}\label{56}
    \dot{x}=2|t+\tau|x(t),\qquad\tau\in \mathbb{R}.
\end{equation}
By a similar argument as the previous example, we can verify that this family has a uniform quadratic dichotomy, with parameters $(0;1,*,1)$.

\smallskip

Moreover, this family of systems also has uniform exponential dichotomy, and in the exponential case, for any $\beta>0$ there is some $K\geq 1$ such that the parameters can be chosen as $(0;K,*,\beta)$. Nevertheless, this family \eqref{56} does not present exponential growth, but only quadratic growth, with constants $L=1$ and $a=1$.
    }
\end{example}

Now we will see an instance of a nonuniform dichotomy for a family of systems in which the rate itself is preserved, but the parameters cannot be fix for all the equations.

\begin{example}\label{0560}
{\rm   Fix $0<3\eta<\lambda$.  Consider the family of differential equations
\begin{equation}\label{560}
    \dot{x}=-\left(\lambda+\eta(t+\tau)\sin(t+\tau)\right)x(t),\qquad\tau\in \mathbb{R}.
\end{equation}

For $\tau=0$, the equation \eqref{560} has nonuniform exponential dichotomy with parameters $(1;e^{2\eta},-\lambda+\eta,*,2\eta,*)$.  Similarly, it has nonuniform exponential bounded growth with constants $L=e^{2\eta}$, $a=\lambda+\eta$ and $\epsilon=2\eta$.

\smallskip

Therefore, by Lem.~\ref{44}, for every $\tau \in \mathbb{R}$, the equation \eqref{560} exhibits a nonuniform exponential dichotomy (since the translations of the exponential rate remain exponential) with parameters $(1; e^{6|\tau|\eta + 2\eta}, -\lambda + \eta, *, 2\eta, *)$.

\smallskip

Similarly, by Lem.~\ref{101},  for every $\tau \in \mathbb{R}$, the equation \eqref{560} exhibits a nonuniform exponential growth, with constants $L_\tau=e^{6|\tau|\eta + 2\eta}$, $a=\lambda+\eta$ and $\epsilon=2\eta$.

\smallskip

Note that there is no way to chose the same parameters for every equation on \eqref{560}, thus $K$ and $L$ must be non-constant maps. 
}
\end{example}

In the remainder of this section, we explore basic properties of $\NmuD$. To this end, let us recall some fundamental concepts. We denote by $\pi: \mathbb{R} \times \mathbb{R}^N \to \mathbb{R}$ the canonical projection, defined by $\pi(\tau, x) = \tau$. Given a subset $\W \subset \mathbb{R} \times \mathbb{R}^N$ and $\tau \in \pi\W$, we define the fiber of $\W$ over $\tau$ as $\W(\tau):=\{x\in \mathbb{R}^N : (\tau,x)\in \W\}.$

\smallskip

A subbundle is a subset $\W\subset\mathbb{R}\times\mathbb{R}^N$ verifying that for every $\tau\in \pi\W$, the fiber $\W(\tau)$ is a vector subspace of $\mathbb{R}^N$, and all these subspaces have the same dimension. A subbundle $\W\subset \mathbb{R}\times \mathbb{R}^N$ is called invariant if $\Phi(t,\tau)x\in \W$ for every $(\tau,x)\in \W$ and $t\in \mathbb{R}$. Given two subbundles $\W$, $\V$, we define their intersection and sum respectively as
$$\V\cap \W:=\{(\tau,x)\in \mathbb{R}\times\mathbb{R}^N:x\in \V(\tau)\cap\W (\tau)\},$$
$$\V+\W:=\{(\tau,x+y)\in \mathbb{R}\times\mathbb{R}^N:x\in \V(\tau)\land y\in \W(\tau)\}.$$
If, moreover, they satisfy $\V\cap \W=\pi\V\times\{0\}$, then write $\V+\W=\V\oplus\W$. The following are called respectively the stable, unstable and bounded invariant subbundles:
\begin{align}\label{33}
    \S&:=\left\{(\tau,x)\in \mathbb{R}\times \mathbb{R}^n: \sup_{t\geq 0}\norm{\Phi(t,\tau)x}< +\infty\right\},\nonumber\\
        \U&:=\left\{(\tau,x)\in \mathbb{R}\times \mathbb{R}^n: \sup_{t\leq 0}\norm{\Phi(t,\tau)x}< +\infty\right\},\\
                \B&:=\left\{(\tau,x)\in \mathbb{R}\times \mathbb{R}^n: \sup_{t\in \mathbb{R}}\norm{\Phi(t,\tau)x}< +\infty\right\}.\nonumber
\end{align}

On the other hand, given an invariant projector $\P : \mathbb{R} \to \mathbb{R}^{N \times N}$, we define its range and null space respectively as
\begin{align*}
    \rang\P&:=\{(\tau,x)\in \mathbb{R}\times \mathbb{R}^N: \P(\tau)x=x\},\\
    \ker\P&:=\{(\tau,x)\in \mathbb{R}\times \mathbb{R}^N: \P(\tau)x=0\}.
\end{align*}

The goal of the following result is to characterize how, in the presence of a dichotomy, the stable and unstable subbundles fully determine the family of projections. By employing standard techniques \cite{Siegmund2002}, the reader can prove this result as an exercise.

\begin{proposition}\label{300}
    Consider the family of systems $\F$, and suppose it has $\NmuD$ with parameters $(\P;K,\alpha,\beta,\theta,\nu)$. Then, we have
\begin{itemize}
        \item [i)] $\B=\mathbb{R}\times\{0\}$,
        \item [ii)] $\rang\P=\S$ and $\ker\P=\U$,
        \item [iii)] $\S\cap\U=\mathbb{R}\times\{0\}$ and  $\S\oplus\U=\mathbb{R}\times\mathbb{R}^N$. 
    \end{itemize}
\end{proposition}

\subsection{Dichotomy spectrum}\label{dynaspectrum}

In this subsection we explore the spectral structure of dichotomies for families of translated systems. Fix a locally integrable map $A:\mathbb{R}\to \mathbb{R}^{N\times N}$ and a growth rate $\mu$. Then, for  $\gamma\in \mathbb{R}$ consider the \textbf{$(\mu,\gamma)$-shifted equation} from \eqref{01}
\begin{equation}\label{001}
\dot{y}=\left(A(t)-\gamma\,\frac{\dot{\mu}(t)}{\mu(t)}\,\Id\right)y(t),\quad t\in \mathbb{R},
\end{equation}
as well as the $\tau$-translated equation from \eqref{001}
\begin{equation}\label{003}
        \dot{y}=\left(A_\tau(t)-\gamma\,\frac{\dot{\mu_\tau}(t)}{\mu_\tau(t)}\,\Id\right)y(t),\quad t\in \mathbb{R}.
 \end{equation}

Analogously as before, we will consider the family of all such $\tau$-translated equations, denoted by $\FT$. We denote the evolution operator od \eqref{001} by $\Phi_\tau^{\mu,\gamma}(\cdot,\cdot)$.

\smallskip

Following the convention introduced in \cite{Rasmussen,Rasmussen2} we say that the family of systems $\FT$ has $\NmuD$ for $\gamma=+\infty$ if there are $\widehat{\gamma}\in\mathbb{R}$, $\alpha<0$, $\theta\geq 0$ and $K:\mathbb{R}\to[1,+\infty)$ such that the family of equations $\FT$ has $\NmuD$ with parameters $(\Id;K,\alpha,*,\theta,*)$. Analogously, we say that the family $\FT$ has $\NmuD$ for $\gamma=-\infty$ if there are $\widehat{\gamma}\in\mathbb{R}$, $\beta>0$, $\nu\geq 0$ and $K:\mathbb{R}\to [1,+\infty)$ such that the the family of equations $\FT$ has $\NmuD$ with parameters $(0;K,*,\beta,*,\nu)$. An analogous convention applies for the single equation \eqref{001}. In order to present the following definitions, we introduce the extended real line $\overline{\mathbb{R}}=\mathbb{R}\cup\{+\infty,-\infty\}$ with its usual topology and order.

\begin{definition}
The \textbf{nonuniform $\mu$-dichotomy spectrum of the family of systems $\F$} is defined as
$$\Sigma_\NmuD(A_{\tau\in \mathbb{R}}):=\{\gamma\in \overline{\mathbb{R}}: \text{the family }\FT\text{ does not have } \NmuD\},$$
and $\rho_\NmuD(A_{\tau\in \mathbb{R}}):=\overline{\mathbb{R}}\setminus\Sigma_\NmuD(A_{\tau\in \mathbb{R}})$ is the \textbf{nonuniform $\mu$-resolvent set}.
\end{definition}

Similarly, we define the \textbf{uniform $\mu$-dichotomy spectrum} $\Sigma_\muD(A_{\tau\in \mathbb{R}})$ and the \textbf{uniform $\mu$-resolvent set} $\rho_\muD(A_{\tau\in \mathbb{R}})$. The connected components of $\rho_\NmuD$ and $\rho_\muD$ are called spectral gaps, while the connected components of $\Si_\NmuD$ and $\Si_\muD$ are called spectral intervals. We keep this nomenclature on other spectra we define later.

\smallskip

An analogous concept is established on \cite[p.~623]{Silva} for the single equation \eqref{01}. We introduce minor modifications to that notion and obtain that the \textbf{nonuniform $\mu$-dichotomy spectrum} of \eqref{01} is given by
$$\Sigma_\NmuD(A):=\{\gamma\in \overline{\mathbb{R}}: \eqref{001} \text{ does not have } \NmuD\},$$
and analogously we define $\rho_\NmuD(A)$, $\Sigma_\muD(A)$ and $\rho_\muD(A)$. 

\smallskip

In \cite[Thm.~8]{Silva}, the author proves a spectral theorem, concluding that $\Si_\NmuD(A)$ is a finite union of closed intervals, and is bounded if \eqref{01} presents $\Nmu$-growth. By combining this result with Cor.~\ref{200} we obtain the following spectral theorem for dichotomies of families of translated systems.

\begin{theorem}\label{126}
    (Spectral Theorem) Consider the family of systems $\F$ and a growth rate $\mu$. There exists some $m\in \{1,\dots,N\}$ such that
		\[
		\Si_\NmuD(A)=\Si_\NmuD(A_{\tau\in \mathbb{R}})=\bigcup_{i=1}^m[a_i,b_i],
		\]
		for some $a_i,b_i\in \overline{\mathbb{R}}$ verifying $a_i\leq b_i<a_{i+1}$. Moreover, the family of systems $\F$ has $\Nmu$-growth if and only if $\Sigma_\NmuD(A_{\tau\in \mathbb{R}})$ is bounded.
\end{theorem}

An analogous theorem is proved for the uniform case.

\section{Comparison and translations of growth rates}\label{comparison}

In this section we explore results about translations of growth rates. These results will be fundamental to the developments of Sect.~\ref{limit}. Moreover, the conclusions from this section show that, for a large class of growth rates, it is not possible to write families of dichotomies and bounded growth with a unified growth rate. To this end, let us consider the comparison criteria for growth rates introduced in \cite{GJ2}.

\begin{definition}
    Consider two growth rates $\mu,\sigma:\mathbb{R}\to \mathbb{R}^+$.  We say $\mu$ is \textbf{faster} than $\sigma$ ($\mu\gg\sigma$) if for every $\alpha,\walpha<0$ there exist $M\geq 1$ such that
    \begin{equation}\label{803}
        \left(\frac{\mu(t)}{\mu(s)}\right)^\alpha\leq M\left(\frac{\sigma(t)}{\sigma(s)}\right)^{\walpha},\quad \forall\,t\geq s.
    \end{equation}
  Equivalently,  we say that $\sigma$ is \textbf{slower} than $\mu$ ($\sigma\ll\mu$).

\smallskip
  
  On the other hand, we say $\mu$ is \textbf{weakly faster} than $\sigma$ ($\mu\succ\sigma$) if there exist $M\geq 1$ such that
    \begin{equation*}
        \frac{\sigma(t)}{\sigma(s)}\leq M \frac{\mu(t)}{\mu(s)},\quad \forall\,t\geq s.
    \end{equation*}
Equivalently, we say that $\sigma$ is \textbf{weakly slower} than $\mu$ ($\sigma\prec\mu$).
\end{definition}

The reader can verify that the polynomial rate $p(t)=(|t|+1)^{\sgn(t)}$ is, for instance, slower than the exponential rate. That is, for every $\alpha,\walpha<0$, there is some $M$ such that, in particular:
$$e^{\alpha(t-s)}\leq M\left(\frac{1+t}{1+s}\right)^{\walpha},\quad \forall\,t\geq s\geq 0.$$

Similarly, the quadratic growth rate $\qeg(t)=e^{\sgn(t)t^2}$ is faster than the exponential. That is, for every $\alpha,\walpha<0$, there is some $M$ such that, in particular:
$$e^{\alpha(t^2-s^2)}\leq M e^{\walpha(t-s)},\quad \forall\,t\geq s\geq 0.$$

On the other hand, if $\mu$ is a growth rate, then $\widetilde{\mu}(t)=\mu(t)^2$ is a new growth rate and $\widetilde{\mu}$ is weakly faster (but not faster) than $\mu$.

\smallskip

The following is immediate from the definition.

\begin{remark}\label{404}
    {\rm
        Consider the equation \eqref{01} and two growth rates $\sigma$ and $\mu$, such that $\sigma\prec \mu$, then:
    \begin{itemize}
        \item [i)] if \eqref{01} has $\mu$-dichotomy, then it also has $\sigma$-dichotomy,
        \item [ii)] if \eqref{01} has $\sigma$-growth, then it also has $\mu$-growth.
    \end{itemize}
    }
\end{remark}

The reader can prove the following as an exercise.

\begin{lemma}\label{409}
    For any $r_1,r_2\in (0,1)$ and $\widetilde{r}_1,\widetilde{r}_2\in (1,+\infty)$, with $r_1<r_2$ and $\widetilde{r}_1<\widetilde{r}_2$, we have the following order between the polynomial, subexponential, exponential and superexponential growth rates:
    $$p\ll\ueg_{r_1}\ll\ueg_{r_2}\ll\exp\ll\seg_{\widetilde{r}_1}\ll\seg_{\widetilde{r}_2}$$
\end{lemma}

Now we delve to study how this comparison criteria relate to translations of growth rates. As a first remark, note that for a growth rate $\mu$ and $\tau\in \mathbb{R}$ we have $$\mu\gg\exp\Rightarrow\mu_\tau\gg\exp\quad\text{and}\quad \exp\gg\mu\Rightarrow\exp\gg\mu_\tau.$$

In the following, we study the limits of translations of nonexponential growth rates.

\begin{lemma}\label{401}
    Choose $r\in (0,1)$ and a growth rate $\mu\prec \ueg_r$. Then, the maps $t\mapsto \limsup_{\tau\to \pm\infty}\mu_{\tau}(t)$ and $t\mapsto \liminf_{\tau\to \pm\infty}\mu_{\tau}(t)$ are bounded and bounded away from zero.
\end{lemma}

\begin{proof}
By hypothesis, there is some $M\geq 1$ such that
$$ \frac{\mu(t)}{\mu(s)}\leq M \frac{\ueg_r(t)}{\ueg_r(s)},\quad \forall\,t\geq s.$$

Let us examine four cases. If $t,\tau>0$ we have $t+\tau>\tau$ and
    $$1\leq \mu_\tau(t)=\frac{\mu(t+\tau)}{\mu(\tau)}\leq M\frac{\ueg_r(t+\tau)}{ \ueg_r(\tau)}=Me^{(t+\tau+1)^r-(\tau+1)^r}\xrightarrow{\tau\to+\infty}M,$$
thus
$$1\leq \liminf_{\tau\to +\infty}\mu_\tau(t)\leq \limsup_{\tau\to+\infty}\mu_\tau(t)\leq M.$$

If $t>0$ and $\tau<0$, then $t+\tau>\tau$ and for $|\tau|$ large enough we have
    $$1\leq \mu_\tau(t)=\frac{\mu(t+\tau)}{\mu(\tau)}\leq M\frac{\ueg_r(t+\tau)}{\ueg_r(\tau)}=Me^{-(1-\tau-t)^r+(1-\tau)^r}\xrightarrow{\tau\to-\infty}M,$$
thus
$$1\leq \liminf_{\tau\to -\infty}\mu_\tau(t)\leq \limsup_{\tau\to-\infty}\mu_\tau(t)\leq M.$$

If $t,\tau<0$ we have $t+\tau<\tau$ and
$$1\leq(\mu_\tau(t))^{-1}=\frac{\mu(\tau)}{\mu(t+\tau)}\leq M\frac{\ueg_r(\tau)}{\ueg_r(t+\tau)}=Me^{-(1-\tau)^r+(1-\tau-t)^r}\xrightarrow{\tau\to-\infty}M,$$
thus
$$M^{-1}\leq \liminf_{\tau\to-\infty}\mu_\tau(t)\leq\limsup_{\tau\to-\infty}\mu_\tau(t)\leq 1.$$

If $t<0$ and $\tau>0$, then $t+\tau<\tau$ and for large enough $\tau$ we have
$$1\leq(\mu_\tau(t))^{-1}=\frac{\mu(\tau)}{\mu(t+\tau)}\leq M\frac{\ueg_r(\tau)}{\ueg_r(t+\tau)}=Me^{(1+\tau)^r-(1+\tau+t)^r}\xrightarrow{\tau\to+\infty}M,$$
thus
$$M^{-1}\leq \liminf_{\tau\to+\infty}\mu_\tau(t)\leq \limsup_{\tau\to+\infty}\mu_\tau(t)\leq 1.$$

Then, combining these cases, the statement follows.
\end{proof}

Now we study the analogous result for superexponential growth rates.

\begin{lemma}\label{402}
    Choose $r\in (1,+\infty)$ and a growth rate $\mu\succ \seg_r$. Then, we have
             \begin{equation*}
    \lim_{\tau\to \pm\infty}\mu_{\tau}(t)= \left\{ \begin{array}{lcc}
             +\infty &  \text{ if } &   t> 0, \\
             1 & \text{ if } &   t= 0, \\
            0 & \text{ if }& t<0 .
             \end{array}
   \right.
\end{equation*}
\end{lemma}

\begin{proof}

By hypothesis, there is some $M\geq 1$ such that
$$ \frac{\mu(t)}{\mu(s)}\geq M^{-1} \frac{\seg_r(t)}{\seg_r(s)},\quad \forall\,t\geq s.$$
Let us examine four cases. If $t,\tau>0$ we have $t+\tau>\tau$ and

    $$\mu_\tau(t)=\frac{\mu(t+\tau)}{\mu(\tau)}\geq M^{-1}\frac{\seg_r(t+\tau)}{ \seg_r(\tau)}=Me^{(t+\tau)^r-\tau^r}\xrightarrow{\tau\to+\infty}+\infty,$$
thus
$$ \liminf_{\tau\to +\infty}\mu_\tau(t)= \limsup_{\tau\to+\infty}\mu_\tau(t)=+\infty.$$

If $t>0$ and $\tau<0$, then $t+\tau>\tau$ and for $|\tau|$ large enough we have
    $$\mu_\tau(t)=\frac{\mu(t+\tau)}{\mu(\tau)}\geq M^{-1}\frac{\seg_r(t+\tau)}{\seg_r(\tau)}=M^{-1}e^{-|t+\tau|^r+|\tau|^r}\xrightarrow{\tau\to-\infty}+\infty,$$
thus
$$\liminf_{\tau\to -\infty}\mu_\tau(t)= \limsup_{\tau\to-\infty}\mu_\tau(t)=+\infty.$$

If $t,\tau<0$ we have $t+\tau<\tau$ and
$$(\mu_\tau(t))^{-1}=\frac{\mu(\tau)}{\mu(t+\tau)}\geq M^{-1}\frac{\seg_r(\tau)}{\seg_r(t+\tau)}=M^{-1}e^{-|\tau|^r+|t+\tau|^r}\xrightarrow{\tau\to-\infty}+\infty,$$
thus
$$ \liminf_{\tau\to-\infty}\mu_\tau(t)=\limsup_{\tau\to-\infty}\mu_\tau(t)=0.$$

If $t<0$ and $\tau>0$, then $t+\tau<\tau$ and for large enough $\tau$ we have
$$(\mu_\tau(t))^{-1}=\frac{\mu(\tau)}{\mu(t+\tau)}\geq M^{-1}\frac{\seg_r(\tau)}{\seg_r(t+\tau)}=M^{-1}e^{\tau^r-(t+\tau)^r}\xrightarrow{\tau\to+\infty}+\infty$$
thus
$$\liminf_{\tau\to+\infty}\mu_\tau(t)=\limsup_{\tau\to+\infty}\mu_\tau(t)=0.$$

Then, combining these cases, the statement follows.
\end{proof}

Taking these results in consideration, we have the following definition.

\begin{definition}\label{608}
    A growth rate $\mu$ is called \textbf{slow} if $\mu\prec \ueg_r$, for some subexponential rate $r\in (0,1)$, while it is called \textbf{fast} if $\mu\succ\seg_{\widetilde{r}}$, for some superexponential rate $\widetilde{r}\in (1,+\infty)$.
\end{definition}

It is immediate that if $\mu$ is slow, then $\mu \ll \exp$, and if $\mu$ is fast, then $\mu \gg \exp$. However, the converses of these implications do not hold. While not every growth rate is classified as either slow or fast, we have nonetheless encompassed a broad class of growth rates, including many of the commonly studied ones. For instance, cubic and quadratic growth are superexponential, square root growth is subexponential, and polynomial growth, in particular, is slow. From Lem.~\ref{401} and Lem.~\ref{402} we deduce the following consequence.

\begin{corollary}\label{405}
    If $\mu$ is a slow growth rate, then there is no growth rate $\sigma$ that verifies $\sigma\prec\mu_\tau$ for every $\tau\in \mathbb{R}$. Similarly, if $\mu$ is fast, then there is no growth rate $\sigma$ satisfying $\sigma\succ\mu_\tau$ for every $\tau\in \mathbb{R}$.
\end{corollary}

As noted in Rem.~\ref{404}, the weak comparison provides a useful tool for transferring dichotomies and bounded growth properties from one growth rate to another. From Cor.~\ref{405} we arrive at the following conclusion:

\begin{corollary}\label{403}
Consider the family of systems $\F$. 
    \begin{itemize}
    \item [i)] If it has $\mu$-growth, for some slow $\mu$, then every equation of the family has exponential bounded growth with the same parameters.
    \smallskip
    \item [ii)] If it has $\mu$-growth, for some fast $\mu$, then it is impossible to describe the growth of all the equations of the family with a unified growth rate.
    \smallskip
    \item [iii)] If it has $\mu$-dichotomy, for some slow $\mu$, then it is impossible to describe the dichotomies of all the equations of the family with a unified growth rate.
    \smallskip
    \item [iv)] If it has $\mu$-dichotomy, for some fast $\mu$, then every equation of the family it has exponential dichotomy with the same parameters.
\end{itemize}
\end{corollary}

\section{Behavior at limiting equations of skew-products}\label{behavior}

In this section we present the main results of this work. Subsection \ref{skew} is an introduction to the skew-product setting, and the main conclusions are derived on Subsection \ref{limit}, proving that growth rates that are either fast or slow either do not occur or are inconsequential in the skew-product context. On the remainder of the section we obtain similar conclusions for the nonuniform exponential case, and we revisit the examples from Section \ref{dynamic} with the skew-product notations.

\subsection{The skew-product setting}\label{skew}

In this subsection we present the basic notions from the skew-product setting, going over the classic notions of exponential dichotomy, bounded growth and dichotomy spectrum.

\smallskip

Consider a space $\Omega$ of locally integrable functions defined on $\mathbb{R}$ taking values in $\mathbb{R}^{N\times N}$, endowed with a topology such that the base flow of translations $\mathbb{R}\times \Omega\to \Omega$, $(\tau,\omega)\mapsto \omega_\tau$, where $\omega_\tau(t)=\omega(t+\tau)$, is well defined and continuous. Tipically, the chosen topology makes $\Omega$ a compact space. Moreover, the considered topology verifies that the evaluation map $A:\Omega\to \mathbb{R}^{N\times N}$, defined by $A(\omega)=\omega(0)$, is continuous. We refer the reader to \cite{Artstein1,Longo,Longo2,Sell68} for several constructions of such topologies.

\smallskip

Given a point $\omega\in\Omega$, its orbit is the set $\mathcal{O}(\omega)=\{\omega_\tau:\tau\in \mathbb{R}\}$. Then define the alpha limit set $\A(\omega)$ via $\widehat{\omega}\in \A(\omega)$ if there is a sequence $(\tau_n)_{n\in \mathbb{N}}$ with $\tau_n\searrow -\infty$ such that $\widehat{\omega}=\lim_{n\to \infty}\omega_{\tau_n}$, in the topology of $\Omega$. Analogously, the omega limit set $\O(\omega)$ is defined as those $\widehat{\omega}\in \Omega$ for which there is a sequence $(\tau_n)_{n\in \mathbb{N}}$ with $\tau_n\nearrow\infty$ such that $\widehat{\omega}=\lim_{n\to \infty}\omega_{\tau_n}$.  Finally, the hull of $\omega$, denoted by $\H(\omega)$, is the closure on $\Omega$ of the orbit of $\omega$. Note that $\H(\omega)$ is the smallest closed invariant set containing $\omega$.

\smallskip

A key characteristic of the considered topologies is that if $\widehat{\omega}=\lim_{n\to \infty}\omega_{\tau_n}$ for some sequence $(\tau_n)_{n\in \mathbb{N}}$ with $\tau_n\searrow -\infty$ or $\tau_n\nearrow\infty$, then $\widehat{\omega}(t)=\lim_{n\to \infty}\omega_{\tau_n}(t)$ for almost every $t\in \mathbb{R}$. In particular, convergence in $\Omega$ implies point-wise convergence almost everywhere.

\smallskip

Employing the evaluation map $A$ we define the family of linear differential systems
\begin{equation}\label{2}
\dot{x}(t)=A(\omega_t)x(t),\qquad \omega\in \Omega.
\end{equation}

Note that a family of translated equations (as those studied in Section \ref{dynamic}) corresponds to the restriction of the family \eqref{2} to the elements of a particular orbit. The fact that $\Omega$ contains only locally integrable functions imply the existence of solutions for \eqref{2}. For a given $\omega\in \Omega$, we denote the evolution operator of \eqref{2} by $\Phi_\omega:\mathbb{R}\times \mathbb{R}\to \mathbb{R}^{N\times N}$, extending the notations from Sect. \ref{dynamic}. In particular, the map $t\mapsto\Phi_\omega(t,0)$ corresponds to the principal fundamental matrix of \eqref{2}. For every $t,s,\tau\in \mathbb{R}$ and $\omega \in \Omega$ we have:
\begin{itemize}
\item $\Phi_\omega$ is non-singular and $\Phi^{-1}_\omega(t,0)=\Phi_{\omega_t}(-t,0)$,
\item $\Phi_\omega(s+t,0)=\Phi_{\omega_t}(s,0)\Phi_\omega(t,0)$, $\Phi_\omega(0,0)=\Id$,
\item $\Phi_\omega(t+\tau,s+\tau)=\Phi_{\omega_\tau}(s,t)$.
\end{itemize}

Moreover, the assignation $(\omega,t,s)\mapsto \Phi_\omega(t,s)$ is jointly continuous. Through this operator, we define the linear skew-product flow
\begin{equation}\label{1}
    \begin{array}{ccc}
         \Psi:\mathbb{R}\times \Omega\times \mathbb{R}^N &\to&\Omega\times \mathbb{R}^N\\
         (t,\omega,x)&\mapsto&\left(\omega_t,\Phi_\omega(t,0)x\right).
    \end{array}
\end{equation}

Consider an invariant set $\Delta\subset \Omega$. In the context of skew-products, an \textbf{invariant projector} is a map $\P:\Delta\to \mathbb{R}^{N\times N}$ such that 
$$\P^2(\omega)=\P(\omega),\quad\P(\omega_t)\Phi_\omega(t,s)=\Phi_\omega(t,s)\P(\omega_s),\qquad\forall\,t,s\in \mathbb{R},\,\omega\in \Delta.$$

It is immediate that if $\P:\Delta\to \mathbb{R}^{N\times N}$ is an invariant projector, then for every $\omega\in \Delta$, the map $s\mapsto\P_\omega(s):=\P(\omega_s)$ is an invariant projector for the equation $\dot{x}=A(\omega_t)x(t)$.

\smallskip

    Conversely, if $s\mapsto \P(s)$ is an invariant projector for the equation $\dot{x}=A(\omega_t)x(t)$, with which it has $\NmuD$, then the map $\mathcal{O}\ni\omega_\tau\mapsto\P(\tau)$ defines an invariant projector for the skew-product over the orbit of $\omega$. The following result shows that, under a bounding condition, this invariant projector is continuously extended to the whole hull. We skip the proof because it presents minor changes with respect to the one of \cite[Prop.~4.3]{Longo}.

\begin{lemma}\label{97}
    Fix $\omega\in \Omega$ and suppose that there is an invariant projector $s\mapsto\P(s)\in \mathbb{R}^{N\times N}$ with which the system $\dot{x}=A(\omega_t)x(t)$ has $\Nmu$-dichotomy for some growth rate $\mu$. If this invariant projector is uniformly bounded, that is
$$\sup_{\tau\in \mathbb{R}}\left\{\norm{\P(\tau)},\norm{\Id-\P(\tau)}\right\}<+\infty,$$
  then, there is a unique continuous map of projections $\P:\H(\omega)\to \mathbb{R}^{N\times N}$ that extends $\P(\omega_\tau)=\P(\tau)$.  
\end{lemma}

The previous lemma is in particular valid when the dichotomy on the point $\omega$ is uniform, since in that case, it is immediate from Lem.~\ref{44} that the projections are bounded. 

\smallskip

The following concepts \cite{Dueñas,Longo,Sacker} depict how classic exponential dichotomies have been studied in the context of skew-product systems. 

\begin{definition}\label{099}
We say that the linear skew-product \eqref{1} admits \textbf{exponential dichotomy} over $\Omega$ if there is a continuous invariant projector $\P:\Delta\to \mathbb{R}^{N\times N}$ and constants $K\geq 1$, $\alpha<0$ and $\beta>0$ such that for every $\omega\in \Omega$ we have
       \[\left\{ \begin{array}{lc}
            \norm{\Phi_\omega(t,s)\P(\omega_s)}&\leq Ke^{\alpha(t-s)}\quad\text{ for }t\geq s,\nonumber\\
            \norm{\Phi_\omega(t,s)[\Id-\P(\omega_s)]}&\leq Ke^{\beta(t-s)}\text{ for }t\leq s.
            \end{array}
           \right.\]
On the other hand, we say that the linear skew-product \eqref{1} has \textbf{exponential bounded growth} if there constants $a>0$, $L\geq 1$ and $\epsilon\geq 0$ such that 
\begin{equation*}
    		\|\Phi_\omega(t,s)\|\leq Le^{a|t-s|},\quad\forall\,t,s\in \mathbb{R}, \,\omega\in \Omega.
\end{equation*}
\end{definition}

Given $\gamma\in \mathbb{R}$, we define the $\gamma$-shifted skew-product of \eqref{1} by
\begin{equation}\label{5}
    \begin{array}{ccc}
         \Psi_\gamma:\mathbb{R}\times \Omega\times \mathbb{R}^N &\to&\Omega\times \mathbb{R}^N\\
         (t,\omega,x)&\mapsto&\left(\omega_t,e^{-\gamma t}\cdot\Phi_\omega(t,0)x\right).
    \end{array}
\end{equation}
With this, we define the \textbf{exponential dichotomy spectrum} of the skew-product \eqref{1} as
$$\Sigma_{\exp\!\mathrm{D}}(\Omega)=\{\gamma\in \overline{\mathbb{R}}: \eqref{5}\text{ has no exponential dichotomy}\}.$$

Within the skew-product framework, a natural question is whether the concept of dichotomy extends beyond the exponential setting. When the base consists of a single orbit, the results of Section~\ref{dynamic} already address this case.   The main difficulties arise instead in the analysis of more general limiting equations (those associated with elements of the omega and alpha limit sets). In the subsections that follow we show that such extension is generally not relevant.


\smallskip


\subsection{Fast and slow behavior}\label{limit}

In the following, we study the consequences of the results of the previous sections. We focus our attention in how fast and slow behaviors interact with the limit sets of a skew-product. These results lead us to the conclusions that both fast and slow behaviors are either inexistent or irrelevant in the context of compact-base skew-products. This is why in the previous section, the concepts of dichotomy, bounded growth and spectrum are only presented for the exponential case.

\begin{theorem}\label{602}
    Consider the skew-product \eqref{1} and chose $\omega\in \Omega$. Suppose the equation \eqref{2} has $\mu$-dichotomy for some fast growth rate (equivalently, $0\not\in \Sigma_\muD(\omega)$). Then, we have $ \mathbb{A}(\omega)=\mathbb{O}(\omega)=\varnothing$.
\end{theorem}

\begin{proof}
 By hypothesis, and employing Lem.~\ref{44}, there are constants $\alpha,\beta>0$ and $K\geq 1$, and a continuous map of projections $\P:\mathcal{O}\to \mathbb{R}^{N\times N}$ such that for every $\tau\in \mathbb{R}$ we have
     \[\left\{ \begin{array}{lc}
            \norm{\Phi_{\omega_\tau}(t,s)\P\left((\omega_\tau)_s\right)}&\leq K\left(\frac{\mu_\tau(t)}{\mu_\tau(s)}\right)^\alpha\,\,\,\quad\text{ for }t\geq s,\nonumber\\
            \norm{\Phi_{\omega_\tau}(t,s)[\Id-\P\left((\omega_\tau)_s\right)]}&\leq K\left(\frac{\mu_\tau(t)}{\mu_\tau(s)}\right)^\beta\qquad\text{ for }t\leq s.
            \end{array}
           \right.\]

In particular, as $\mu$ is fast, there is some superexponential rate $\seg_r$, for some $r\in (1,+\infty)$, such that $\mu\succ\seg_r$, which in particular implies $\mu_\tau\succ(\seg_r)_\tau$ for every $\tau\in \mathbb{R}$. This implies that there is some $M\geq 1$ such that for every $\tau\in \mathbb{R}$ we have

     \[\left\{ \begin{array}{lc}
            \norm{\Phi_{\omega_\tau}(t,s)\P\left((\omega_\tau)_s\right)}&\leq KM\left(\frac{\seg_r(t+\tau)}{\seg_r(s+\tau)}\right)^\alpha\,\,\,\quad\text{ for }t\geq s,\nonumber\\
            \norm{\Phi_{\omega_\tau}(t,s)[\Id-\P\left((\omega_\tau)_s)\right)]}&\leq KM\left(\frac{\seg_r(t+\tau)}{\seg_r(s+\tau)}\right)^\beta\qquad\text{ for }t\leq s.
            \end{array}
           \right.\]

Now, consider the continuous extension $\P:\mathbb{H}(\omega)\to\mathbb{R}^{N\times N}$ of the maps of projections, given by Lem.~\ref{97}. If $\widehat{\omega}\in\mathbb{O}(\omega)$, there is a sequence $(\tau_n)_{n\in\mathbb{R}}$, with $\tau_n\nearrow +\infty$, such that $\widehat{\omega}=\lim_{n\to \infty}\omega_{\tau_n}$. Then, by taking limits in the above estimations, in the same fashion as Lem.~\ref{402} we obtain
     \[\left\{ \begin{array}{lc}
            \norm{\Phi_{\widehat{\omega}}(t,s)\P(\widehat{\omega}_s)}&= 0\,\,\,\quad\text{ for }t> s,\nonumber\\
            \norm{\Phi_{\widehat{\omega}}(t,s)[\Id-\P(\widehat{\omega}_s)]}&=0\qquad\text{ for }t< s,
            \end{array}
           \right.\]
which contradicts that $\vPhi$ is non-singular. Thus, such $\widehat{\omega}$ cannot exist. Analogously we prove $\A(\omega)=\varnothing$.
\end{proof}

\begin{corollary}\label{611}
    If for some $\omega\in \Omega$, the equation \eqref{2} has $\mu$-dichotomy, for some fast growth rate (equivalently, $0\not\in \Sigma_\muD(\omega)$), then $\Omega$ is not compact.
\end{corollary}

\begin{proof}
    If the alpha and omega limit sets are empty, then the sequences of translated growth rates $\{\omega_{\tau_n}\}_{n\in\mathbb{N}}$, for $\tau_n\searrow -\infty$ or $\tau_n\nearrow +\infty$, have no converging subsequences, which implies that $\Omega$ is not compact.
\end{proof}

\begin{corollary}\label{605}
    If for some $\omega\in \Omega$, the equation \eqref{2} has $\mu$-dichotomy, for some fast growth rate (equivalently, $0\not\in \Sigma_\muD(\omega)$), then $\mathcal{O}(\omega)$ is homeomorphic to $\mathbb{R}$ and coincides with $\mathbb{H}(\omega)$.
\end{corollary}

Now we aim to study the complementary results for slow growth rates. 

\begin{theorem}\label{601}
    Consider the skew-product \eqref{1} and chose $\omega\in \Omega$. Suppose the equation \eqref{2} has $\mu$-growth, for some slow growth rate (equivalently, $\pm\infty\not\in \Sigma_\muD(\omega)$). Then, for any $\widehat{\omega}\in \mathbb{A}(\omega)\cup\mathbb{O}(\omega)$, the equation
    \begin{equation}\label{406}
        \dot{x}=A(\widehat{\omega}_t)x(t),
    \end{equation}
    does not have a dichotomy.
\end{theorem}

\begin{proof}
   By hypothesis and Lem.~\ref{101}, there are constants $a>0$ and $L\geq 1$ such that
$$	\|\Phi_{\omega_\tau}(t,s)\|\leq L\left(\frac{\mu_\tau(t)}{\mu_\tau(s)}\right)^{\sgn(t-s)a},\quad\forall\,t,s\in \mathbb{R},\,\forall\,\omega_\tau\in \mathcal{O}(\omega).$$

Then, by the continuity of both the base flow and $\Psi$, and using Lem.~\ref{401}, there is some $M\geq 1$ such that
    $$\norm{\Phi_{\widehat{\omega}}(t,s)}\leq LM,\qquad\forall\,t,s\in \mathbb{R},\,\forall\,\widehat{\omega}\in \A({\omega})\cup\mathbb{O}({\omega}),$$
    hence all the solutions of \eqref{406} are bounded, which implies, by Prop.~\ref{300}, that there is no dichotomy.
\end{proof}

\begin{corollary}\label{617}
    If for some $\omega\in \Omega$, the equation \eqref{2} has $\mu$-growth, for some slow growth rate (equivalently, $\pm\infty\not\in \Sigma_\muD(\omega)$), then there is no dichotomy over any closed invariant set containing $\omega$.
\end{corollary}

Moreover, arguing analogously to Prop.~\ref{128} (see below), we have the following result.

\begin{corollary}
        If for some $\omega\in \Omega$, the equation \eqref{2} has $\mu$-growth, for some slow growth rate (equivalently, $\pm\infty\not\in \Sigma_\muD(\omega)$), then $\Sigma_{\exp\!\mathrm{D}}(\H(\omega))=\{0\}$.
\end{corollary}

\begin{remark}\label{extend}
{\rm
For a general function $\omega: \mathbb{R} \to \mathbb{R}^{N \times N}$, it is not obvious whether there exists a growth rate $\mu$ such that equation \eqref{01} exhibits $\mu$-growth, not even $\Nmu$-growth. Even when a $\mu$-growth is present, it is not guaranteed that $\mu$ falls into one of the types we have studied: exponential, slow, or fast.

\smallskip

If we restrict our attention to those $\omega$ for which equation \eqref{01} exhibits $\mu$-growth with $\mu$ being either exponential, slow, or fast, then Thm.~\ref{601} characterizes the case of slow growth, while Thm.~\ref{602} addresses the case of fast dichotomy. Still, these results do not exhaust all possibilities: an equation may fail to exhibit slow growth and yet also not admit any form of fast dichotomy. In other words, the cases we have analyzed are not exhaustive. Nonetheless, they capture a broad and relevant class of systems — especially given that on one hand, many classical results require both bounded growth and dichotomy in order to derive conclusions, and on the other hand, most of the nonexponential behaviors present in current literature are either slow or fast.
}
\end{remark}


\subsection{Bounded, periodic and almost periodic functions}\label{periodic}

Bounded continuous functions defined on $\mathbb{R}$—and especially those that are almost periodic—possess the notable property that they can always be continuously extended to a compactification of $\mathbb{R}$. By Gelfand’s theory \cite[Chap.~4]{Pedersen}, if $\omega \in \Omega$ is bounded, then there exists a suitable function space and a compatible metric topology in which the orbit $\mathcal{O}(\omega)$ is precompact.

\smallskip

For almost periodic functions \cite{Bohr}, one can use the Bohr compactification \cite{Sacker2}; for arbitrary bounded continuous functions, the general setting is covered by the Stone–Čech compactification. In both cases, the limit sets $\A(\omega)$ and $\O(\omega)$ are nonempty, and the associated hull $\H(\omega)$ consists entirely of bounded functions. This is why bounded continuous functions—and in particular, almost periodic ones—are of special interest when studying skew-product flows. Traditionally, such equations have been studied through the lens of exponential dichotomy and growth. In this subsection, we examine how our framework and results apply in this classical setting.

\begin{proposition}\label{128}
    Choose a fast growth rate $\mu$ and ${\omega}\in \Omega$ such that \eqref{2} has exponential bounded growth. Then, $\Sigma_{\NmuD}(\hat{\omega})=\Sigma_{\muD}(\hat{\omega})=\{0\}$ for every $\hat{\omega}\in \H(\omega)$.
\end{proposition}

\begin{proof}
If \eqref{2} has exponential bounded growth with constants $L\geq 1$ and $a> 0$, then every equation on the skew-product \eqref{1} also has exponential bounded growth with the same constants. Then, by the continuity of the base flow, this bounded growth propagates to the complete hull $\H(\omega)$, {\it i.e.}
    $$\norm{\Phi_{\widehat{\omega}}(t,s)}\leq e^{a|t-s|},\qquad\forall\,t,s\in \mathbb{R},\,\forall\,\widehat{\omega}\in \H(\omega).$$

Choose $\gamma>0$. Then, there is some $M\geq 1$ such that $e^{a(t-s)}\leq M  \left(\frac{\mu_\omega(t)}{\mu_\omega(s)}\right)^{\gamma/2}$ for every $\omega\in \mathcal{O}(\hat{\omega})$. Then, we have
    $$\norm{\Phi^{\mu,\gamma}_{\widehat{\omega}}(t,s)}\leq M \left(\frac{\mu(t)}{\mu(s)}\right)^{-\gamma/2},\qquad\forall\,t\geq s,\,\forall\,\widehat{\omega}\in \H({\omega}),$$
    which states that the $(\mu,\gamma)$-shifted equation has $\mu$-dichotomy with projector $\Id$.  Analogously it is proved that for every $\gamma<0$ the $(\mu,\gamma)$-shifted equation has $\mu$-dichotomy with projector $0$. Finally, by Thm.~\ref{126} we obtain the conclusion.
\end{proof}

Note that the above proposition applies to all bounded or essentially bounded functions, and in particular to continuous Bohr's almost periodic maps, which are always bounded. We now turn to examining how such equations interact with slow growth rates. We begin by addressing the case of periodic functions.

\begin{proposition}\label{613}
        If $\omega$ is a piece-wise continuous periodic function and $\mu$ is a slow growth rate, then the only possible spectral intervals in $\Si_\muD(\widehat{\omega})$ and $\Sigma_\NmuD(\widehat{\omega})$ are $\{0\}$, $\{+\infty\}$ and $\{-\infty\}$ for every $\widehat{\omega}\in \H(\omega)$.
\end{proposition}

\begin{proof}
We prove that $\{0\}$, $\{+\infty\}$ and $\{-\infty\}$ are the only spectral gaps for the uniform spectrum, and the result for the nonuniform case follows from $\Sigma_\NmuD\subset \Sigma_\muD$. First, note that  every $\widehat{\omega}\in \H(\omega)$ is also a piece-wise continuous periodic function, hence by Floquet's theorem \cite{Floquet}, $\dot{x}=\widehat{\omega}(t)x(t)$ is kinematically similar to an autonomous equation $\dot{x}=Bx$. If this equation has nontrivial bounded solutions, then the former also has them, thus $0\in \Si_\muD(\omega)$.  

   \smallskip

   On the other hand, if $\dot{x}=Bx$ has no bounded solutions, then $\dot{x}=\widehat{\omega}(t)x(t)$ has exponential dichotomy, since they share the exponential dichotomy spectrum. Suppose the parameters of the dichotomy are $(\P;K,\alpha,\beta)$. 

   \smallskip

   Let $\gamma\in \mathbb{R}$ and choose $\walpha<\min\{\gamma,0\}$ and $\wbeta>\max\{\gamma,0\}$. Suppose first that $0\neq\P\neq \Id$. As $\mu$ is slow, there is $M\geq 1$ such that
 \[   e^{\alpha(t-s)}\leq M\left(\frac{\mu(t)}{\mu(s)}\right)^{\widetilde{\alpha}},\quad \forall\,t\geq s\qquad \text{and} 
\qquad e^{\beta(t-s)}\leq M\left(\frac{\mu(t)}{\mu(s)}\right)^{\widetilde{\beta}},\quad \forall\,s\geq t.\]

    Then, we have
    \begin{align*}
        \norm{\Phi^{\mu,\gamma}_{\widehat{\omega}}(t,s)\P(s)}&=\norm{\Phi_{\widehat{\omega}}(t,s)\P(s)}\cdot\left(\frac{\mu(t)}{\mu(s)}\right)^{-\gamma}\\
        &\leq  KM\left(\frac{\mu(t)}{\mu(s)}\right)^{\widetilde{\beta}-\gamma},\quad \forall\,t\leq s,
    \end{align*}
    and similarly
    \begin{align*}
\norm{\Phi^{\mu,\gamma}_{\widehat{\omega}}(t,s)[\Id-\P(s)]}\leq KM\left(\frac{\mu(t)}{\mu(s)}\right)^{\widetilde{\alpha}-\gamma},\quad \forall\,t\geq s,
    \end{align*}
    hence the $(\mu,\gamma)$-shifted equation has $\mu$-dichotomy, or equivalently $\gamma\in \rho_{\mu\mathrm{D}}(\widehat{\omega})$. Since $\gamma\in \mathbb{R}$ was arbitrary,  $0\neq\P\neq \Id$ and the projectors associated to dichotomies are unique, then $\Sigma_{\mu\mathrm{D}}(\widehat{\omega})\subset \{\pm\infty\}$. Finally, if $\P=\Id$, an analogous argument shows $\Sigma_\muD(\widehat{\omega})=\{-\infty\}$, and if $\P=0$ then $\Sigma_\muD(\widehat{\omega})=\{+\infty\}$.
\end{proof}

The conclusions from Prop.~\ref{613} also apply to a larger class of equations, those for which always one, and only one, of the following statements is true: i) either has exponential dichotomy, or, ii) there are nontrivial bounded solutions. Note, however, that the result above does not fully capture the range of behaviors that can arise from almost periodic maps, as these may exhibit more intricate dynamics. To the best of our knowledge, the only attempt at developing a Floquet-type theory in the almost periodic setting was made in \cite{Russel}, restricted to the two-dimensional case.

\subsection{The nonuniform exponential case}

So far, we have described—at least to some extent—the behavior of limit points in skew-product systems for both fast and slow growth rates. We now focus on the nonuniform exponential case, which remains less explored compared to the classical (uniform) exponential regime that has been extensively studied in the literature. In this context, we obtain the following characterization:

\begin{theorem}\label{615}
Choose $\omega\in \Omega$. If \eqref{2} does not present uniform exponential growth (equivalently, $\pm\infty\in\Sigma_{\exp\!\mathrm{D}}(\omega)$), then $\mathbb{A}(\omega)=\mathbb{O}(\omega)=\varnothing$. 
\end{theorem}

\begin{proof}
We know $\omega$ is locally integrable. If it were uniformly locally integrable, that is, if there were some $M>0$ such that 
    $$\frac{1}{t}\int_\tau^{\tau+t}\norm{\omega(s)}ds\leq M,\qquad\forall\,t,\tau\in \mathbb{R},$$
then \eqref{2} would have uniform exponential growth. Therefore, we could find some $t_0>0$ and a sequence $(\tau_n)_{n\in \mathbb{N}}$, verifying $\tau_n\searrow -\infty$ or $\tau_n\nearrow +\infty$, such that
    $$\lim_{n\to \infty}\frac{1}{t_0}\int_{\tau_n}^{\tau_n+t_0}\norm{\omega(s)}ds=\lim_{n\to \infty}\frac{1}{t_0}\int_{0}^{t_0}\norm{\omega_{\tau_n}(s)}ds=+\infty,$$
    then, such $\widehat{\omega}=\lim_{n\to \infty}\omega_{\tau_n}$ is not locally integrable, therefore not in $\Omega$.
\end{proof}

An immediate consequence follows from the previous result.

\begin{corollary}\label{616}
  Choose $\omega\in \Omega$. If \eqref{2} does not present uniform exponential growth (equivalently, $\pm\infty\in\Sigma_{\exp\!\mathrm{D}}(\omega)$), then $\mathcal{O}(\omega)=\H(\omega)$ is homeomorphic to $\mathbb{R}$. In particular, $\Omega$ is not compact.
\end{corollary}

Note that Cor.~\ref{616} is also valid if \eqref{2} has a nonuniform exponential growth that cannot be written as a uniform exponential growth. In other words, if $\Sigma_{\mathrm{N}\!\exp\!\mathrm{D}}(\omega)$ is bounded but $\Sigma_{\exp\!\mathrm{D}}(\omega)$ is not bounded, then there are no limiting equations.


\subsection{Revisited examples}

We conclude this work by revisiting Examples \ref{550}, \ref{056} and \ref{0560} from Sect.~\ref{dynamic}, reformulated in the skew-product context, to illustrate our conclusions.
 In the following, the spaces $\Omega$ are families of functions defined on $\mathbb{R}$, taking values on $\mathbb{R}$ (since these are scalar examples). These spaces of functions are given the compact-open topology \cite[p.~301]{Bourbaki}.

\begin{example}
{\rm Consider $\Omega=\{t\mapsto \frac{1}{1+|t+\tau|}:\tau\in \mathbb{R}\}\cup\{t\mapsto 0\}$.  Note that this $\Omega$ homeomorphic to the unit circle $\mathbb{S}^1$, thus compact. This skew-product parametrizes the family of differential equations \eqref{55}, and $\dot{x}=0$. As the family of equations \eqref{55} has polynomial dichotomy and growth (as seen on Ex.~\ref{550}), then there is $p_\tau$-dichotomy and growth for every equation in the orbit $\mathcal{O}:=\{t\mapsto \frac{1}{1+|t+\tau|}:\tau\in \mathbb{R}\}$, where $p_\tau$ are the translations of the polynomial growth rate.

\smallskip

Note, however, there is no dichotomy over the whole $\Omega$, since on the other orbit $\A(\omega_0)=\O(\omega_0)$, where $\omega_0(t)=\frac{1}{1+|t|}$, the associated equation is $\dot{x}=0$, which has no kind of dichotomy, as every solution is bounded (Prop.~\ref{300}).
}    
\end{example}

\begin{example}
{\rm
 Consider $\Omega=\{t\mapsto |t+\tau|:\tau\in \mathbb{R}\}$.  This skew-product parametrizes the family of equations \eqref{56}. As this family of equations has quadratic dichotomy and growth (Ex.~\ref{056}), then the skew-product has $\qeg_\tau$-dichotomy and growth, where $\qeg_\tau$ are the translations of the quadratic growth rate.

\smallskip
 
Finally, note that $\Omega$ is not compact, since it is homeomorphic to the real line $\mathbb{R}$. Moreover, by taking point-wise limits of translations of any $\omega\in \Omega$, for sequences $(\tau_n)_{n\in \mathbb{N}}$ with $\tau_n\searrow -\infty$ or $\tau_n\nearrow\infty$, the conclusion is $\lim_{n\to \infty}\omega_{\tau_n}(t)=+\infty$ for every $t\in \mathbb{R}$. Hence, even if we take $\Omega$ as a subset of a larger space of functions, we obtain $\H(\omega)=\mathcal{O}(\omega)$ for every $\omega\in \Omega$, that is, there are no limit points.
}
\end{example}

\begin{example}
{\rm 
 Fix $0<3\eta<\lambda$. Consider $\Omega=\{t\mapsto -(\lambda+\eta(t+\tau)\sin(t+\tau)):\tau\in \mathbb{R}\}$. This skew-product parametrizes the family of equations \eqref{560}, which has nonuniform exponential dichotomy and growth, as seen on Ex.~\ref{0560}. On the other hand, as $\Omega$ is homeomorphic to $\mathbb{R}$, then there is nonuniform exponential dichotomy and growth in the sense of skew-products.
} 
\end{example}

\end{document}